\documentclass[a4paper,12pt,reqno]{amsart}

\usepackage{amsfonts}
\usepackage{amsmath}
\usepackage{amssymb}

\usepackage{mathrsfs}
\usepackage{hyperref}

\numberwithin{equation}{section}
\theoremstyle{plain}
\newtheorem{thm}{Theorem}[section]

\newtheorem{cor}[thm]{Corollary}

\theoremstyle{definition}

\def\Rn{{\mathbb R}^n}
\def\X{\mathbb X}

\begin{document}

\title[Hardy inequalities on metric measure spaces]
{Hardy inequalities on metric measure spaces}

\dedicatory{Dedicated to G. H. Hardy on the occasion of 100 years of his famous inequality}

\author[Michael Ruzhansky]{Michael Ruzhansky}
\address{
  Michael Ruzhansky:
  \endgraf
  Department of Mathematics
  \endgraf
  Imperial College London
  \endgraf
  180 Queen's Gate, London, SW7 2AZ
  \endgraf
  United Kingdom
  \endgraf
and
  \endgraf
  Department of Mathematics: Analysis, Logic and Discrete Mathematics
  \endgraf
  Ghent University, Belgium
  \endgraf
  and 
  \endgraf
  School of Mathematical Sciences
    \endgraf
    Queen Mary University of London
  \endgraf
  United Kingdom
  \endgraf
  {\it E-mail address} {\rm ruzhansky@gmail.com}
  }
  
\author[Daulti Verma]{Daulti Verma}
\address{
  Daulti Verma:
  \endgraf
    Miranda House College
  \endgraf
  University of Delhi
  \endgraf
  Delhi 110007
  \endgraf
  India
  \endgraf
   and
  \endgraf
    Department of Mathematics
  \endgraf
  Imperial College London
  \endgraf
  180 Queen's Gate, London, SW7 2AZ
  \endgraf
  United Kingdom
  \endgraf
  {\it E-mail address} {\rm daulti.verma@mirandahouse.ac.in}
 }

%\thanks{The first author was supported in parts by the EPSRC grant EP/R003025/1 and by the Leverhulme Grant RPG-2017-151. No new data was collected or generated during the course of research.}

\date{\today}

\subjclass{26D10, 22E30.} \keywords{Hardy inequalities, metric measure spaces, homogeneous group, hyperbolic space, Riemannian manifolds with negative curvature}

\begin{abstract}
In this note we give several characterisations of weights for two-weight Hardy inequalities to hold on general metric measure spaces possessing polar decompositions. Since there may be no differentiable structure on such spaces, the inequalities are given in the integral form in the spirit of Hardy's original inequality. We give examples obtaining new weighted Hardy inequalities on $\Rn$, on homogeneous groups, on hyperbolic spaces, and on Cartan-Hadamard manifolds. We note that doubling conditions are not required for our analysis.
\end{abstract}

\maketitle

\tableofcontents
   
   \section{Introduction}
   
In \cite{Hardy:1920}, Hardy showed his famous inequality 
\begin{equation}\label{EQ:Hardy0}
\int_b^\infty \left(\frac{\int_b^x f(t) dt}{x}\right)^p dx \leq \left(\frac{p}{p-1}\right)^p \int_b^\infty f(x)^p dx,
\end{equation} 
where $p>1$, $b>0$, and $f\geq 0$ is a nonnegative function. The original discrete version of this inequality goes back to \cite{Hardy1919} making this year the 100${}^{th}$ anniversary of this topic, see \cite{RS-book} for a historical discussion. Consequently, we do not try to provide a comprehensive historical account here but refer to \cite{RS-book} for an extensive historical overview of the subject.

In particular, since then a lot of work has been done on Hardy inequalities in different forms and in different settings. It is clearly impossible to give a complete overview of the literature, so let us only refer to
books and surveys by Opic and Kufner \cite{OK-book}, Davies \cite{Davies}, 
Kufner, Persson and Samko \cite{Hardy-weighted-book,KP-Samko},
Edmunds and Evans \cite{DE-book}, Mazya \cite{Mazya-Sobolev-spaces,Mazya-book}, Ghoussoub and Moradifam \cite{GM-book}, Balinsky, Evans and Lewis \cite{BEL-book}, and references therein. Hardy inequalities in 1D with weights have been also studied in \cite{AM,M} in a similar spirit to our approach. Another list of conditions was given in \cite{DV2}, however, both in more limited settings, as well as, more recently, in \cite{GKP10}. More precisely, 
in our Theorem \ref{THM:Hardy1} in the special case $\mathbb X=\mathbb R$, the conditions $D_i$ correspond to conditions $A_i(s)$, $s=1,2,3,4$, in \cite[Theorem 2]{GKP10}. 

In this paper we show that the inequality \eqref{EQ:Hardy0} actually holds in a much more general setting, also with rather general pairs of weights. However, the weights have to satisfy certain compatibility conditions for such inequalities to hold true, and these conditions are {\em necessary and sufficient}.
   
We note that the only known cases of our results are essentially only the Euclidean ones. The examples we give on homogeneous groups and hyperbolic spaces are new, but the main result itself is of course more general, completely characterising the weights for the integral Hardy inequality on metric measure spaces. 
The importance of such results is, in particular, in that they lead to {\em a variety of hypoelliptic Hardy-Sobolev and other inequalities, once we apply it with the weights associated to Riesz kernels (for hypoelliptic operators}, see \cite{RY-hypoelliptic}).

  More specifically, we consider metric spaces $(\mathbb X,d)$ with a Borel measure $dx$ allowing for the following {\em polar decomposition} at $a\in{\mathbb X}$: we assume that there is a locally integrable function $\lambda \in L^1_{loc}$  such that for all $f\in L^1(\mathbb X)$ we have
   \begin{equation}\label{EQ:polar}
   \int_{\mathbb X}f(x)dx= \int_0^{\infty}\int_{\Sigma_r} f(r,\omega) \lambda(r,\omega) d\omega_r dr,
   \end{equation}
    for the sets $\Sigma_r=\{x\in\mathbb X: d(x,a)=r\}$ %\subset \mathbb X$ 
    with a measure on it denoted by $d\omega=d\omega_r$. %, and $(r,\omega)\rightarrow a $ as $r\rightarrow0$. 
    The condition \eqref{EQ:polar} is rather general since we allow the function $\lambda$ to depend on the whole variable $x=(r,\omega)$. 
%Moreover, there is no assumption on any relation between the metric on $\mathbb X$ and the geometry of the sets $\Sigma_r$, as we are just assuming the formula \eqref{EQ:polar} for some sets $\Sigma_r$ and some function $\lambda(r,\omega)$.

The reason to assume \eqref{EQ:polar} is that since $\X$ does not have to have a differentiable structure, the function $\lambda(r,\omega)$ can not be in general obtained as the Jacobian of the polar change of coordinates. However, if such a differentiable structure exists on $\X$, the condition \eqref{EQ:polar} can be obtained as the standard polar decomposition formula. 
In particular, let us give several examples of $\mathbb X$ for which the condition \eqref{EQ:polar} is satisfied with different expressions for $ \lambda (r,\omega)$:

\begin{itemize}
\item[(I)] Euclidean space $\Rn$: $ \lambda (r,\omega)= {r}^{n-1}.$
\item[(II)] Homogeneous groups: $ \lambda (r,\omega)= {r}^{Q-1}$, where $Q$ is the homogeneous dimension of the group. Such groups have been consistently developed by Folland and Stein \cite{FS-Hardy}, see also an up-to-date exposition in \cite{FR}.
\item[(III)] Hyperbolic spaces $\mathbb H^n$:  $\lambda(r,\omega)=(\sinh {r})^{n-1}$.
\item[(IV)] Cartan-Hadamard manifolds: Let $K_M$ be the sectional curvature on $(M, g).$ A Riemannian manifold $(M, g)$ is called {\em a Cartan-Hadamard manifold} if it is complete, simply connected and has non-positive sectional curvature, i.e., the sectional curvature $K_M\le 0$ along each plane section at each point of M. Let us fix a point $a\in M$ and denote by 
$\rho(x)=d(x,a)$ the geodesic distance from $x$ to $a$ on $M$. The exponential map ${\rm exp}_a :T_a M \to  M$ is a diffeomorphism, see e.g. Helgason \cite{DV3}.  Let $J(\rho,\omega)$ be the density function on $M$, see e.g. \cite{DV1}. Then we have the following polar decomposition: 
$$
\int_M f(x) dx=\int_0^{\infty}\int_{\mathbb S^{n-1}}f({\rm exp}_{a}(\rho \omega))J(\rho,\omega) \rho^{n-1}d\rho d\omega,
$$
so that we have \eqref{EQ:polar} with $\lambda(\rho,\omega)= J(\rho,\omega) \rho^{n-1}.$ 
 \item[(V)] Complete manifolds: Let $M$ be a complete manifold. Let $p\in M$ and let $C(p)$ denote the cut locus of $p$. Let $D_{p}:=M\backslash C(p)$ and $S(p;r):=\{x\in M_{p}:|x|=r\}$, where $|\cdot|$ is the Riemannian length, where  $M_{p}$ stands for the tangent space  at $p$. Then for any $p\in M$ and any integrable function $f$ on $M$ we have (e.g. see \cite[Formula III.3.5, P.123]{Cha06}) the polar decompsition
        \begin{equation}\label{pol_decom_manif}
        \int_{M}f dV=\int_{0}^{+\infty}dr\int_{r^{-1}S(p;r)\cap D_{p}}f(\exp r\xi)\sqrt{g}(r;\xi)d\mu_{p}(\xi)
        \end{equation}
        for some function $\sqrt{g}$ on $D_{p}$, where $r^{-1}S(p,r)\cap D_{p}$ is the subset of $S_{p}$ obtained by dividing each of the elements of $S(p,r)\cap D_{p}$ by $r$, and $S_{p}:=S(p;1)$. Here $d\mu_{p}(\xi)$ is the Riemannian measure on $S_{p}$ induced by the Euclidean Lebesgue measure on $M_{p}$. We refer to \cite{Cha06}, \cite[Chapter 4]{Li12} and \cite[Chapter 1, Paragraph 12]{CLN06} for more details on this decompsition.
\end{itemize}

%\smallskip
 Throughout this paper, by $ A\approx B$ we will always mean that the expressions $A$ and $B$ are equivalent.
 
% \smallskip
% The authors would like to thank Nurgissa Yessirkegenov for discussions and for checking our calculations. 

\section{Main results}

We denote by $B(a,r)$ the ball in $\mathbb X$ with centre $a$ and radius  $r$, i.e 
$$B(a,r):= \{x\in\mathbb X : d(x,a)<r\},$$ 
where $d$ is the metric on $\X$. 
Once and for all we will fix some point $a\in {\mathbb X}$, and we will write $${\vert x \vert}_a := d(a,x).$$

Our first main result is the following characterisation of weights $u$ and $v$ for the corresponding Hardy inequality to hold on $\X$, with the characterisation for the conjugate Hardy inequality given in Theorem \ref{THM:Hardy2}. The first condition is the Muckenhoupt condition while the other ones are equivalent to it.

\begin{thm}\label{THM:Hardy1}
Let $1<p\le q <\infty$ and let $s>0$. Let $\mathbb X $ be a metric measure space with a polar decomposition \eqref{EQ:polar} at a. 
Let $u,v> 0$ be measurable functions  positive a.e in $\mathbb X$  such that $u\in L^1(\mathbb X\backslash \{a\})$ and $v^{1-p'}\in L^1_{loc}(\mathbb X)$. Denote
\begin{align}
U(x):= { \int_{\mathbb X\backslash{B(a,|x|_a )}} u(y) dy}  \nonumber
\end{align} 
and 
\begin{align} 
V(x):= \int_{B(a,|x|_a  )}v^{1-p'}(y)dy\nonumber. 
\end{align}
Then the inequality
\begin{equation}\label{EQ:Hardy1}
\bigg(\int_\mathbb X\bigg(\int_{B(a,\vert x \vert_a)}\vert f(y) \vert dy\bigg)^q u(x)dx\bigg)^\frac{1}{q}\le C\bigg\{\int_{\mathbb X} {\vert f(x) \vert}^pv(x)dx\bigg\}^{\frac1p}
\end{equation}
holds for all measurable functions $f:\X\to{\mathbb C}$ if and only if  any of the following equivalent conditions hold:

\begin{enumerate}
\item $\mathcal D_{1} :=\sup_{x\not=a} \bigg\{U^\frac{1}{q}(x) V^\frac{1}{p'}(x)\bigg\}<\infty.$
\end{enumerate}

\begin{enumerate}\setcounter{enumi}{1}
\item $\mathcal D_{2}:=\sup_{x\not=a} \bigg\{\int_{\mathbb X\backslash{B(a,|x|_a )}}u(y)V^{q(\frac{1}{p'}-s)}(y)dy\bigg\}^\frac{1}{q}V^s(x)<\infty.$
\item $\mathcal D_{3}:=\sup_{x\not=a}\bigg\{\int_{B(a,|x|_a)}u(y)V^{q(\frac{1}{p'}+s)}(y)dy\bigg\}^{\frac{1}{q}}V^{-s}(x)<\infty $, provided that $u,v^{1-p'}\in L^1(\X)$.
\end{enumerate}

\begin{enumerate}\setcounter{enumi}{3}
\item $\mathcal D_{4}:=\sup_{x\not=a}\bigg\{\int_{B(a,\vert x \vert_a)}v^{1-p'}(y) U^{p'(\frac{1}{q}-s)}(y)dy\bigg\}^\frac{1}{p'}U^s(x)<\infty.$ 

\item $\mathcal D_{5}:=\sup_{x\not=a}\bigg\{\int_{\mathbb X\backslash{B(a,\vert x \vert_a )}}v^{1-p'}(y)U^{p'(\frac{1}{q}+s)}(y)dy\bigg\}^\frac{1}{p'}U^{-s}(x)<\infty$, provided that $u,v^{1-p'}\in L^1(\X)$.
\end{enumerate}

Moreover, the constant $C$ for which \eqref{EQ:Hardy1} holds and quantities $\mathcal D_{1}-\mathcal D_{5}$ are related by 
\begin{equation}\label{EQ:constants}
\mathcal D_{1} \leq C\leq \mathcal D_1(p')^{\frac{1}{p'}} p^\frac{1}{q},
\end{equation}   
and 
$$\mathcal D_1 \le \left(\max(1,{p'}{s})\right)^\frac{1}{q}\mathcal D_2, \;\mathcal D_2 \le (\max(1,\frac{1}{p's}))^\frac{1}{q} \mathcal D_1,$$ 
$$(\frac{sp'}{1+p's})^\frac{1}{q} \mathcal D_3 \le \mathcal D_1\le(1+sp')^\frac{1}{q}\mathcal D_3,$$
$$\mathcal D_1 \le (\max(1,qs))^\frac{1}{p'} \mathcal D_4,\; \mathcal D_4 \le (\max(1,\frac{1}{qs}))^\frac{1}{p'}
\mathcal D_1,$$ 
$$(\frac{sq}{1+qs})^\frac{1}{p'}\mathcal D_5 \le \mathcal D_1 \le (1+sq)^\frac{1}{p'} \mathcal D_5.$$
\end{thm}

In particular, Theorem \ref {THM:Hardy1} is an extension of \eqref{EQ:Hardy0} to the setting of metric measures spaces $\X$ with the polar decomposition \eqref{EQ:polar}: in particular, for $p=q$ and real-valued nonnegative measurable $f\geq 0$, inequality \eqref{EQ:Hardy1} becomes 
$$
\int_\mathbb X\bigg(\int_{B(a,\vert x \vert_a)} f(y) dy\bigg)^p u(x)dx \le C \int_{\mathbb X} {f(x)}^p v(x)dx,
$$
as an extension of \eqref{EQ:Hardy0}. Indeed, in this case we can take $u(x)=\frac{1}{x^p}$, $v(x)=1$, $\mathbb X=[b,\infty)$, $a=b$, so that Theorem \ref {THM:Hardy1} implies  \eqref{EQ:Hardy0}.

For the results in the case of $\mathbb X=\mathbb R$ we can refer to \cite{Hardy-weighted-book, PS01}, and also to \cite{PSW07} for inequalities for $q<p$.
For $\mathbb X=\Rn$, the result has been proved in \cite {Daulti-Verma:Razmadze}, with related inequalities obtained in one dimension in  \cite{DV2,OK-book}. For related works on hyperbolic spaces we can refer to \cite{DV4,RY18_manif}, and to \cite{DV5,RY18_manif} for inequalities on Cartan-Hadamard manifolds, with the background analysis available in   \cite{DV1,DV3}.
For the analysis of Hardy inequalities on homogeneous groups we can refer to
\cite{RS-identities,RSY18_Tran}.

Let us also briefly discuss the conjugate Hardy inequality to that in Theorem \ref{THM:Hardy1}:

\begin{thm}\label{THM:Hardy2}
Let $1<p\le q <\infty$ and $s>0$. Let $\mathbb X $ be a metric measure space with a polar decomposition as in \eqref{EQ:polar}. Let $u,v >0$ be measurable functions positive a.e. such that
$u\in L^1_{loc}(\mathbb X)$ and  $v^{1-p'}\in L^1(\mathbb X\backslash \{a\})$. Let
\begin{align}
U(x)= {\int_{{B(a,\vert x \vert_a )}} u(y) dy}  \nonumber
\end{align} 
and \begin{align} 
V(x)= \int_{\mathbb X\backslash B(a,\vert x \vert_a )}v^{1-p'}(y)dy\nonumber. 
\end{align}
Then the inequality
\begin{align}\label{EQ:Hardycon}
\bigg(\int_\mathbb X\bigg(\int_{\mathbb X\backslash B(a,\vert x \vert_a)}\vert f(y) \vert dy\bigg)^q u(x)dx\bigg)^\frac{1}{q}\le C\bigg\{\int_{\mathbb X} {\vert f(x) \vert}^p v(x)dx\bigg\}^\frac{1}{p}
\end{align}
holds for all measurable functions f if and only if any of the following equivalent conditions holds:
\begin{enumerate}
\item $\mathcal D_{1}^{*} :=\sup_{x\not=a} \bigg\{U^\frac{1}{q}(x) V^\frac{1}{p'}(x)\bigg\}<\infty$.
\item $\mathcal D_{2}^{*}:=\sup_{x\not=a} \bigg\{\int_{{B(a,\vert x \vert_a )}}u(y)V^{q(\frac{1}{p'}-s)}(y)dy\bigg\}^\frac{1}{q}V^s(x)<\infty.$ %provided that $V^{q(\frac{1}{p'}-s)}(y)$ makes sense.

\item $\mathcal D_{3}^{*}:=\sup_{x\not=a}\bigg\{\int_{\mathbb X\backslash B(a,\vert x \vert_a)}u(y)V^{q(\frac{1}{p'}+s)}(y)dy\bigg\}^{\frac{1}{q}}V^{-s}(x)<\infty,$
provided that $u,v^{1-p'}\in L^1(\mathbb X).$
\item $\mathcal D_{4}^{*}:=\sup_{x\not=a}\bigg\{\int_{\mathbb X\backslash B(a,\vert x \vert_a)}v^{1-p'}(y) U^{p'(\frac{1}{q}-s)}(y)dy\bigg\}^\frac{1}{p'}U^s(x)<\infty.$
%provided that $U^{p'(\frac{1}{q}-s)}(y)$ makes sense.
\item $\mathcal D_{5}^{*}:=\sup_{x\not=a}\bigg\{\int_{{B(a,\vert x \vert_a )}}v^{1-p'}(t)U^{p'(\frac{1}{q}+s)}(t)dt\bigg\}^\frac{1}{p'}U^{-s}(x)<\infty,$
provided that $u,v^{1-p'}\in L^1(\mathbb X).$
\end{enumerate}
\end{thm}

\section{Applications and examples}

In this section we will give examples of the application of Theorem \ref {THM:Hardy1} in the settings of homogeneous groups, hyperbolic spaces, and Cartan-Hadamard manifolds.

\subsection{Homogeneous groups}

Let $\mathbb G$ be a homogeneous group of homogeneous dimension $Q$, equipped with a quasi-norm $|\cdot|$.
For the general description of the setup of homogeneous groups we refer to \cite{FS-Hardy} or \cite{FR}. Particular example of homogeneous groups are the Euclidean space $\Rn$ (in which case $Q=n$), the Heisenberg group, as well as general stratified groups (homogeneous Carnot groups) and graded groups.

In relation to the notation of this paper, let us take $a=0$ to be the identity of the group $\mathbb G$. We can also simplify the notation denoting $\vert x \vert_a$ by $\vert x \vert$, which is consistent with the notation for the quasi-norm $|\cdot|.$

If we take power weights 
$$u(x)= {\vert x \vert}^\alpha \textrm{ and } v(x)= {\vert x \vert }^\beta,$$
then the inequality \eqref{EQ:Hardy1} holds for $1< p \le q <\infty$ if and only if 
$$
\mathcal D_1=\sup_{r>0} \bigg(\displaystyle \sigma\int_{r}^\infty {\rho}^\alpha {\rho}^{Q-1}d\rho \bigg)^\frac{1}{q} \bigg(\displaystyle \sigma \int_{0}^r{\rho}^{\beta(1-p')} {\rho}^{Q-1}d\rho\bigg)^\frac{1}{p'}<\infty,
$$
where $\sigma$ is the area of the unit sphere in $\mathbb G$ with respect to the quasi-norm $|\cdot|.$ For this supremum to be well-defined we need to have  $\alpha+Q<0$
 and $\beta(1-p')+Q>0.$
Consequently, we have
\begin{multline*}
\mathcal D_1=\sigma^{(\frac1q+\frac{1}{p'})}\sup_{r>0}
\bigg(\displaystyle\int_{r}^\infty{\rho}^{\alpha+Q-1}d\rho\bigg)^\frac{1}{q}\bigg(\displaystyle\int_0^{r} {\rho}^{\beta(1-p')+Q-1}d\rho \bigg)^\frac{1}{p'} \\
= \sigma^{(\frac1q+\frac{1}{p'})}
\sup_{r>0} \frac {{r}^ \frac{\alpha+Q}{q}}{{|\alpha+Q|}^\frac{1}{q}}
\frac{ {r}^\frac{\beta(1-p')+Q} {p'}}{{(\beta(1-p')+Q)}^\frac{1}{p'}},
\end{multline*} 
which is finite if and only if the power of $r$ is zero.
Summarising, we obtain

\begin{cor}\label{COR:hom}
Let $\mathbb G$ be a homogeneous group of homogeneous dimension $Q$, equipped with a quasi-norm $|\cdot|$. Let $1<p\le q <\infty$ and let $\alpha,\beta\in\mathbb R$. Then the inequality
\begin{equation}\label{EQ:Hardy1hg}
\bigg(\int_\mathbb G\bigg(\int_{B(0,\vert x \vert)}\vert f(y) \vert dy\bigg)^q {\vert x \vert}^\alpha dx\bigg)^\frac{1}{q}\le C\bigg\{\int_{\mathbb G} {\vert f(x) \vert}^p {\vert x \vert }^\beta dx\bigg\}^{\frac1p}
\end{equation}
holds for all measurable functions $f:\mathbb G\to{\mathbb C}$  if and only if 
$\alpha+Q<0$,  $\beta(1-p')+Q>0$ and 
$\frac{\alpha+Q}{q}+\frac{\beta(1-p')+Q} {p'}=0.$
Moreover, the constant $C$ for \eqref{EQ:Hardy1hg} satisfies
\begin{equation}\label{EQ:constants-hg}
\frac{\sigma^{\frac1q+\frac{1}{p'}}}{|\alpha+Q|^\frac{1}{q}(\beta(1-p')+Q)^\frac{1}{p'}}
\leq C\leq (p')^{\frac{1}{p'}} p^\frac{1}{q}\frac{\sigma^{\frac1q+\frac{1}{p'}}}{|\alpha+Q|^\frac{1}{q}(\beta(1-p')+Q)^\frac{1}{p'}},
\end{equation}  
where $\sigma$ is the area of the unit sphere in $\mathbb G$ with respect to the quasi-norm $|\cdot|.$
\end{cor} 

 \subsection{Hyperbolic spaces} 
 
 Let $\mathbb H^n$ be the hyperbolic space of dimension $n$ and let $a\in \mathbb H^n$.
 Let us take the weights 
 $$u(x)= (\sinh {\vert x \vert_a})^\alpha \textrm{ and } v(x)= (\sinh {\vert x \vert_a})^\beta.$$  
Then, passing to polar coordinates, $\mathcal D_1$ is equivalent to
 
 $$\mathcal D_1\simeq \sup_{\vert x \vert_a>0}\bigg(\displaystyle\int_{\vert x \vert_a}^\infty(\sinh{\rho})^{\alpha+n-1}d\rho\bigg)^\frac{1}{q}\bigg(\displaystyle\int_0^{\vert x \vert_a} (\sinh{\rho})^{\beta(1-p')+n-1}d\rho \bigg)^\frac{1}{p'}.$$
 For the integrability of the first and the second terms we need, respectively,
$\alpha+n-1<0$ and  $ \beta(1-p')+n>0.$

Let us now analyse conditions for this supremum to be finite. 
For $\vert x \vert_a \gg1 $, it can be written as

\begin{multline*}
\sup_{\vert x \vert_a \gg1} \bigg(\displaystyle\int_{\vert x \vert_a}^\infty (\exp{\rho})^{\alpha+n-1}d\rho\bigg)^\frac{1}{q} \bigg(\displaystyle\int_0^{\vert x \vert_a} (\exp{\rho})^{\beta(1-p')+n-1}d\rho \bigg)^\frac{1}{p'}\\
\simeq \sup_{\vert x \vert_a \gg1} (\exp {\vert x \vert_a})^{\bigg(\frac{\alpha+n-1}{q}+\frac{\beta(1-p')+n-1}{p'} \bigg)},
\end{multline*} 
which is finite if and only if $\frac{\alpha+n-1}{q}+\frac{\beta(1-p')+n-1}{p'}\le0.$
For $\vert x \vert_a\ll 1 $, it can be written as

$\sup_{\vert x \vert_a\ll 1} \bigg(\displaystyle\int_{\vert x \vert_a}^\infty(\sinh{\rho})^{\alpha+n-1}d\rho\bigg)^\frac{1}{q} \bigg(\displaystyle\int_0^{\vert x \vert_a} {\rho}^{\beta(1-p')+n-1}d\rho \bigg)^\frac{1}{p'}$

$\simeq \sup_{\vert x \vert_a\ll 1} \bigg(\displaystyle\int_{\vert x \vert_a}^ R (\sinh{\rho})^{\alpha+n-1}d\rho +\displaystyle\int_{R}^\infty(\sinh{\rho})^{\alpha+n-1}d\rho\bigg)^\frac{1}{q} \bigg(\displaystyle\int_0^{\vert x \vert_a} {\rho}^{\beta(1-p')+n-1}d\rho \bigg)^\frac{1}{p'}.$

For some small $R$ we have ${\sinh{\rho} }_{\vert x \vert_a\le \rho<R} \approx \rho$, so that the above supremum is 

$\approx \sup_{\vert x \vert_a\ll 1}\bigg( \vert x \vert_a^{\alpha+n}+C_R\bigg)^\frac{1}{q} $ ${\vert x \vert_a }^\frac{\beta(1-p')+n}{p'}.$

Now, for $\alpha+n\ge0$, this is 

$\approx \sup_{\vert x \vert_a\ll 1}{\vert x \vert_a}^\frac{\beta(1-p')+n}{p'}$, which is finite if and only if $\frac{\beta(1-p')+n}{p'} \ge0$.
At the same time, for for $\alpha+n <0 $ it is

$\approx \sup_{\vert x \vert_a\ll 1}{\vert x \vert_a}^{\frac{\alpha+n}{q}+\frac{\beta(1-p')+n}{p'}}$ , which is finite if and only if $\frac{\alpha+n}{q}+\frac{\beta(1-p')+n}{p'}\ge0$.

Summarising, we obtain the following

\begin{cor}\label{COR:hyp}
Let $\mathbb H^n$ be the hyperbolic space, $a\in \mathbb H^n$, and let $|x|_a$ denote the hyperbolic distance from $x$ to $a$. Let $1<p\le q <\infty$ and let $\alpha,\beta\in\mathbb R$. Then the inequality
\begin{equation*}\label{EQ:Hardy1hyp}
\bigg(\int_{\mathbb H^n}\bigg(\int_{B(a,\vert x \vert_a)}\vert f(y) \vert dy\bigg)^q (\sinh {\vert x \vert_a})^\alpha dx\bigg)^\frac{1}{q}\le C\bigg\{\int_{\mathbb H^n} {\vert f(x) \vert}^p (\sinh {\vert x \vert_a})^\beta dx\bigg\}^{\frac1p}
\end{equation*}
holds for all measurable functions $f:\mathbb H^n\to{\mathbb C}$  if and only if 
the parameters satisfy either of the following conditions
\begin{itemize}
\item[(A)] for $\alpha+n \ge 0 $, if and only if $\alpha+n<1$, $\beta(1-p')+n>0$ and $\frac{\alpha+n}{q}+\frac{\beta(1-p')+n}{p'}\le \frac{1}{q}+\frac{1}{p'}$;
\item[(B)] for $\alpha+n <0 $,  if and only if  $\beta(1-p')+n>0$  and $0\leq \frac{\alpha+n}{q}+\frac{\beta(1-p')+n}{p'}\le \frac{1}{q}+\frac{1}{p'}$.
\end{itemize} 
\end{cor} 

\subsection{Cartan-Hadamard manifolds}

Let $(M,g)$ be a Cartan-Hadamard manifold and assume that the sectional curvature $K_M$ is constant. In this case it is known that $J(t,\omega)$ is a function of $t$ only. More precisely, if  $K_M=-b$ for $b\ge0$, then 
$J(t,\omega)= 1$ if $b=0$, and $J(t,\omega)=(\frac{\sinh \sqrt{b}t}{\sqrt{b}t})^{n-1}$ for $b>0$,
see e.g. \cite{BC01-book} or \cite[p. 166-167]{DV1}. %, as well as \cite{DV5}.
 
 When $b=0$, then let us take $u(x)=  \vert x \vert_a^\alpha$ and $v(x)=  \vert x \vert_a^\beta$ , then the  inequality \eqref{EQ:Hardy1} holds for $1<p\le q<\infty$ if and only if\\
 
 $\sup_{\vert x \vert_a>0} \bigg(\displaystyle\int_{M \backslash B(a,\vert x \vert_a)}\vert y \vert_a^\alpha dy \bigg)^\frac{1}{q} \bigg(\displaystyle \int_{B(a,\vert x \vert_a)}\vert y \vert_a^{\beta(1-p')} dy\bigg)^\frac{1}{p'}<\infty$.\\

After changing to the polar coordinates, this is equivalent to            

$\sup_{\vert x \vert_a>0}\bigg(\int_{\vert x \vert_a }^\infty  {\rho}^{\alpha+n-1} d\rho \bigg)^\frac{1}{q}\bigg(\int_0^{\vert x \vert_a} {\rho}^{\beta(1-p')+n-1}d\rho \bigg)^\frac{1}{p'}$, \\
which is finite if and only if conditions of Corollary \ref{COR:hom} hold with $Q=n$ (which is natural since the curvature is zero).

When $b>0$, let us take $u(x)=(\sinh \sqrt{b}{\vert x \vert_a})^\alpha$  and $v(x)=(\sinh \sqrt{b}{\vert x \vert_a})^\beta$. Then he  inequality \eqref{EQ:Hardy1} holds for $1<p\le q<\infty$ if and only if\\

$\sup_{\vert x \vert_a>0} \bigg(\displaystyle\int_{\mathbb M \backslash B(a,\vert x \vert_a)} (\sinh \sqrt{b}{\vert y \vert_a})^{\alpha} dy \bigg)^\frac{1}{q} \bigg(\displaystyle \int_{B(a,\vert x \vert_a)}(\sinh \sqrt{b}{\vert y \vert_a})^{\beta(1-p')}dy \bigg)^\frac{1}{p'}<\infty.$\\

After changing to the polar coordinates, this supremum is equivalent to 

$\sup_{\vert x \vert_a>0}(\displaystyle\int_{\vert x \vert_a}^\infty (\sinh \sqrt{b}{t})^\alpha (\frac{\sinh \sqrt{b} t}{\sqrt{b} t})^{n-1} t^{n-1}dt )^\frac{1}{q}$\\
$\quad\times(\displaystyle \int_{0}^ {\vert x \vert_a} (\sinh \sqrt{b}{t})^{\beta(1-p')}(\frac{\sinh \sqrt{b} t}{\sqrt{b} t})^{n-1} t^{n-1} dt )^\frac{1}{p'}$\\

$\simeq\sup_{\vert x \vert_a>0}\bigg(\displaystyle\int_{\vert x \vert_a}^\infty (\sinh \sqrt{b}{t})^{\alpha+n-1}dt \bigg)^\frac{1}{q}\bigg(\displaystyle\int_0^{\vert x \vert_a} (\sinh \sqrt{b} t)^{\beta(1-p')+n-1}dt \bigg)^\frac{1}{p'}, $\\
which has the same conditions for finiteness as the case of the hyperbolic space in Corollary \ref{COR:hyp} (which is also natural since it is the negative constant curvature case).

     \section{Equivalence of weight conditions}
     
    In this section we prove that the quantities $\mathcal D_{1}$--$\mathcal D_{5}$
involving the weights in Theorem \ref{THM:Hardy1} are equivalent. However, it will be convenient to formulate it in the following slightly more general form:
     
  \begin{thm}\label{THM:equivalence}
    Let $\alpha , \beta, s >0$  and let
    $f  \in L^{1}(\mathbb X\backslash \{a\})$, $g \in L^{1} _{loc} (\mathbb X)$, be such that $ f,g >0 $ are positive a.e in $\mathbb X$. Let us denote
    \begin{align}
             F(x):= \int_{\mathbb X\backslash B(a,\vert x \vert_a)}f(y)dy\nonumber,
    \end{align}
    and
    \begin{align}
            G(x):=\int_{B(a,\vert x \vert_a)} g(y)dy\nonumber.
    \end{align}
    Then the following quantities are equivalent:
    \begin{enumerate}
    \item$\mathcal A_1:= \sup_{x\not=a}A_1(x;\alpha,\beta):= \sup_{x\not=a}F^\alpha(x)G^\beta(x).$
    \item$\mathcal A_2:=\sup_{x\not=a} A_2(x;\alpha,\beta,s):=\sup_{x\not=a}{\bigg(\int_{\mathbb X\backslash{B(a,\vert x \vert_a)}}f(y)G^{(\beta-s)/\alpha}(y)dy\bigg)^\alpha}   G^s(x),$ provided that $G^{(\beta-s)/\alpha}(y)$ makes sense.
    \item$\mathcal A_3:=\sup_{x\not=a} A_3(x;\alpha,\beta,s):= \sup_{x\not=a}\bigg(\int_{B(a,\vert x \vert_a)}g(y)F^{(\alpha-s)/\beta}(y)dy\bigg)^\beta F^s(x),$ provided that $F^{(\alpha-s)/\beta}(y)$ makes sense.     \item$\mathcal A_4:=\sup_{x\not=a} A_4(x;\alpha,\beta,s):= \sup_{x\not=a}\bigg(\int_{B(a,\vert x \vert_a)}f(y)G^{(\beta+s)/\alpha}(y)dy\bigg)^\alpha G^{-s}(x),$ provided that $f,g\in L^1(\X)$ and that $G^{-s}(x)$ makes sense.
    \item$\mathcal A_5:=\sup_{x\not=a} A_5(x;\alpha,\beta,s):=\sup_{x\not=a}\bigg(\int_{\mathbb X\backslash{B(a,\vert x \vert_a)}}g(y) F^{(\alpha+s)/\beta}(y) dy\bigg)^\beta F^{-s}(x),$
    provided $f,g\in L^1(\X)$ and that that $F^{-s}(x)$ makes sense.
    \end{enumerate}
Moreover, we have the following relations between the above quantities:
\begin{itemize}
\item $\mathcal {A}_1\le (\max(1,\frac{s}{\beta}))^\alpha \mathcal {A}_2 $ and   $\mathcal {A}_2\le  (\max(1,\frac{\beta}{s}))^\alpha \mathcal {A}_1$;
\item  $\mathcal {A}_1\le(\max(1,\frac{s}{\alpha}))^\beta \mathcal {A}_3 $ and  $\mathcal {A}_3\le  (\max(1,\frac{\alpha}{s}))^\beta \mathcal {A}_1$;
\item  $ (\frac{s}{\beta+s})^\alpha \mathcal A_4 \le \mathcal A_1\le (1+\frac{s}{\beta})^\alpha \mathcal A_4 $ and  $ (\frac{s}{\alpha+s})^\beta \mathcal A_5 \le \mathcal A_1\le (1+\frac{s}{\alpha})^\beta \mathcal A_5 $.
\end{itemize}
\end {thm}

   \begin{proof}[Proof of Theorem \ref{THM:equivalence}]
   $\boxed{{\mathcal A_1\approx \mathcal A_2}}$
   
   We will first consider the case $s\le\beta$. Then for $\vert y \vert_a\geq \vert x \vert_a$ we have 
       $  G^{(\beta-s)/\alpha}(y)\ge G^{(\beta-s)/\alpha}(x)$. Consequently, we can estimate
   \begin{align*}
   A_2(x;\alpha,\beta,s) & ={\bigg(\int_{\X\backslash{B(a,\vert x \vert_a)}}f(y)G^{(\beta-s)/\alpha}(y)dy\bigg)^\alpha} G^s(x)\nonumber
 \\
 &    \ge{\bigg(\int_{\X\backslash{B(a,\vert x \vert_a)}}f(y)dy\bigg)^\alpha G^\beta(x)}\nonumber
\\ &
  =F^{\alpha}(x)G^{\beta}(x),\nonumber
   \end{align*}
  which implies  $\mathcal A_2 \ge \mathcal A_1$.
    For $s>\beta$, let us first introduce some notation, using the polar decomposition \eqref{EQ:polar}. First, we denote
     \begin{align*}
    W(x) & :={\int_{\X\backslash{B(a,\vert x \vert_a)}}f(y)G^{(\beta-s)/\alpha}(y)dy}\nonumber\\
           & =\int_{\vert x \vert_a}^{\infty}\int_{\Sigma_r}f(r,\omega) \lambda(r,\omega)\widetilde{G_1}(r)^{(\beta-s)/\alpha} d\omega_r dr\nonumber\\
  & = \int_{\vert x \vert_a}^{\infty}{\widetilde{W}}(r)dr\nonumber\\
  & =:\widetilde{W_1}(\vert x \vert_a),\nonumber
   \end{align*}
   where
$$
   \widetilde{G_1}(r) := \int_0^r\int_{\Sigma_s} g(s,\sigma)\lambda(s,\sigma)d\sigma_s ds =\int_0^r\widetilde{G}(s)ds,
 $$
 with $\widetilde{G}(s):=\int_{\Sigma_s} g(s,\sigma)\lambda(s,\sigma)d\sigma_s$,
  and
$$
  \widetilde{W}(r):=\int_{\Sigma_r}f(r,\omega) \lambda(r,\omega)\widetilde{G_1}(r)^{(\beta-s)/\alpha}d\omega_r.
$$
Moreover, we denote 
$$
   \widetilde{F_1}(r) : = \int_r^{\infty} \int_{\Sigma_s} \lambda(s,\sigma)f(s,\sigma)d\sigma_s ds=\int_r^{\infty}\widetilde{F}(s)ds.
$$

Using the function $W$ defined above, we can estimate
   \begin{align*}
   & F^{\alpha}(x)G^{\beta}(x) \\
   & = G^{\beta}(x)\bigg(\int_{\X\backslash{B(a,\vert x \vert_a)}}f(y)G^{(\beta-s)/\alpha}(y)G^{(s-\beta)/\alpha}(y)W^{(s-\beta)/s}(y)W^{(\beta-s)/s}(y)dy\bigg)^\alpha\nonumber
   \\
   &=G^{\beta}(x)\bigg(\int_{\vert x \vert_a}^{\infty}\int_{\Sigma_r} \lambda(r,\omega) f(r,\omega)\widetilde{G_1}^{(\beta-s)/\alpha}(r)\widetilde{G_1}^{(s-\beta)/\alpha}(r)\widetilde{W_1}^{(s-\beta)/s}(r)\widetilde{W_1}^{(\beta-s)/s}(r)d\omega_r dr\bigg)^\alpha\nonumber
   \\ & 
   \le\bigg(\sup_{r>\vert x \vert_a}\widetilde{G_1}^{(s-\beta)}(r)\widetilde{W_1}^{(s-\beta)\alpha/s}(r)\bigg)G^{\beta}(x)\bigg(\int_{\vert x \vert_a}^{\infty}\widetilde{W}(r) {(\int_r^{\infty}\widetilde{W}(s)ds)}^{(\beta-s)/s}dr\bigg)^\alpha\nonumber
 \\
&= \bigg(\sup_{\vert y \vert _a>\vert x \vert_a}{G}^{s}(y){W}^{\alpha}(y)\bigg)^{(s-\beta)/s}\bigg(\frac{s}{\beta}\bigg)^{\alpha}
G^{\beta}(x) W^{(\beta\alpha)/s}(x)\nonumber
\\ & 
 \le {\bigg(\frac{s}{\beta}\bigg)}^{\alpha}\bigg(\sup_{\vert y \vert_a >\vert x \vert_a}{G}^{s}(y){W}^{\alpha}(y)\bigg)^{(1-{\beta/s})}
\bigg(\sup_{{\vert x \vert_a}>0} G^{s}(x) W^{\alpha}(x)\bigg)^{\beta/s}\nonumber
\\ &   \le {\bigg(\frac{s}{\beta}\bigg)}^{\alpha}\sup_{{\vert x \vert_a} >0} A_2(x;\alpha,\beta,s).\nonumber
   \end{align*} 
  Therefore, we obtain
   $$
 \mathcal A_1 \le {\left(\frac{s}{\beta}\right)}^{\alpha} \mathcal A_2.
$$
Hence, we have for every $s>0$ the inequality
$$
\mathcal A_1 \le {\left(\max(1,\frac{s}{\beta})\right)}^{\alpha} \mathcal A_2.
$$
    Conversely, we have for $s<\beta$,
   \begin{align*}
 & G^s(x)W^\alpha(x) \\
 & =G^s(x)\bigg(\int_{\X\backslash{B(a,\vert x \vert_a)}}f(y)G^{(\beta-s)/\alpha}(y)F^{(\beta-s)/\beta}(y)F^{(s-\beta)/\beta}(y)dy\bigg)^\alpha\nonumber\\
& =G^s(x)\bigg(\int_{\vert x \vert_a}^{\infty}{\int_{\Sigma_r} \lambda(r,\omega)f(r,\omega)}\bigg(\int_0^r\int_{\Sigma_s} \lambda(s,\sigma) g(s,\sigma)ds d\sigma_s\bigg)^{(\beta-s)/\alpha}\nonumber\\
&\quad\times\widetilde{F_1}^{(\beta-s)/\beta}(r)\widetilde{F_1}^{(s-\beta)/\beta}(r)drd\omega_r \bigg)^\alpha\nonumber \\
 & = G^s(x)\bigg(\int_{\vert x \vert_a}^{\infty}{\int_{\Sigma_r} \lambda(r,\omega)f(r,\omega)}\widetilde{G_1}^{(\beta-s)/\alpha}(r)\widetilde{F_1}^{(\beta-s)/\beta}(r)\widetilde{F_1}^{(s-\beta)/\beta}(r)d\omega_r dr\bigg)^\alpha.\nonumber
\end{align*}
  Consequently, we can estimate
   \begin{align*}
 & G^s(x)W^\alpha(x)
 \\ &
 \le\bigg(\sup_{r>\vert x \vert_a}\widetilde{G_1}^{(\beta-s)/\alpha}(r)\widetilde{F_1}^{(\beta-s)/\beta}(r)\bigg)^\alpha G^{s}(x) \bigg(\int_{\vert x \vert_a}^{\infty}\widetilde{F}(r)\left(\int_r^{\infty}\widetilde{F}(s)ds\right)^{(s-\beta)/\beta} dr\bigg)^\alpha\nonumber  
 \\ &=\bigg(\sup_{r>\vert x \vert_a}\widetilde{G_1}^{\beta}(r)\widetilde{F_1}^{\alpha}(r)\bigg)^{(\beta-s)/\beta}  G^{s}(x) {\bigg(\frac{\beta}{s}\bigg)}^\alpha \widetilde{F_1}^{(\alpha s)/\beta}(r)\nonumber 
 \\ &\le\bigg(\sup_{\vert y \vert_a>\vert x \vert_a} G^{\beta}(y) F^{\alpha}(y)\bigg)^{(\beta-s)/\beta}{\bigg(\frac{\beta}{s}\bigg)}^\alpha \bigg(\sup_{\vert x \vert_a>0}G^{\beta}(x) F^{\alpha}(x)\bigg)^{s/\beta}\nonumber
\\   &\le\bigg(\frac{\beta}{s}\bigg)^\alpha \sup_{\vert x \vert_a>0} A_1 (x;\alpha,\beta),\nonumber
  \end{align*}
  which gives   
      $
      \mathcal A_2 \leq {\left(\frac{\beta}{s}\right)}^\alpha \mathcal A_1.
     $
     
   On the other hand, for  $s\geq\beta$, when $\vert y \vert_a > \vert x \vert_a$  we have
   $
   G^{(\beta-s)/\alpha}(y) \le G^{(\beta-s)/\alpha}(x).
   $
   
   Therefore, we can estimate
   \begin{align*}
   A_{2}(x;\alpha,\beta,s) &= G^{s}(x)\bigg(\int_{\X\backslash{B(a,\vert x \vert_a)}}f(y)G^{(\beta-s)/\alpha}(y)dy\bigg)^\alpha\nonumber
 \\ & 
   \le G^{s}(x)\bigg(\int_{\X\backslash{B(a,\vert x \vert_a)}}f(y)dy\bigg)^\alpha G^{(\beta-s)}(x)\nonumber
\\
&=F^\alpha(x)G^\beta(x)\nonumber
\end{align*}
   i.e. $\mathcal A_2 \le \mathcal A_1.$ 
Therefore, we have for $s>0$, the overall estimate
$$
\mathcal A_2 \le {\left(\max(1,\frac{\beta}{s})\right)}^{\alpha} \mathcal A_1.
$$
Hence we have also shown that
$
  \mathcal A_1 \approx \mathcal A_2.
$

  Next we observe that the proof of $\boxed{\mathcal A_1 \approx \mathcal A_3}$ follows along the same lines as that of $\mathcal A_1 \approx \mathcal A_2$, where we just need to interchange the roles of $F$ and $G$.
  
  \smallskip
  $\boxed{\mathcal A_1 \approx \mathcal A_4}$
 
 \smallskip
  Let us denote
  \begin{align*}
W_{0}(x) &:= \int_{B(a,\vert x \vert_a)} f(y) G^{(\beta+s)/\alpha}(y)dy\nonumber
 \\    &= \int_{0}^{\vert x \vert_a}\int_{\Sigma_r} \lambda(r,\omega) f(r,\omega){G}^{(\beta+s)/\alpha}(r,\omega)d\omega_r dr\nonumber\\
\\ 
  &=: \int_{0}^{\vert x \vert_a}\widetilde{W_0}(r)dr,\nonumber
\end{align*}
\\
\\
so that we can write
  $$
   A_4(x;\alpha,\beta,s) = G^{-s}(x)W_0^\alpha(x).
  $$   
  We rewrite
  $A_1$ as
   \begin{align*}
   A_1(x;\alpha,\beta)&=G^{\beta}(x)\bigg(\int_{\mathbb X\backslash{B(a,\vert x \vert_a)}}f(y)G^{(\beta+s)/\alpha}(y)G^{-{(\beta+s)/\alpha}}(y)dy\bigg)^\alpha\nonumber
   \\
   &= G^\beta(x)\bigg(\int_{\vert x \vert_a}^{\infty}\int_{\Sigma_r} \lambda(r,\omega) f(r,\omega){G}^{(\beta+s)/\alpha}(r,\omega)G^{-(\beta+s)/\alpha}(r,\omega) d\omega_r dr\bigg)^\alpha\nonumber
 \\
   &=G^\beta(x)\bigg(\int_{\vert x \vert_a}^{\infty}\widetilde{G_1}^{-(\beta+s)/\alpha}(r)\frac{d}{dr}\bigg(\int_0^r\widetilde{W_0}(s)ds\bigg)dr\bigg)^\alpha.\nonumber
   \end{align*}
  We can estimate this by  
   \begin{align*}
   & \le G^{\beta}(x)\bigg(\widetilde{G_1}^{-(\beta+s)/\alpha}(\infty){W_0}(\infty) + \frac{(\beta+s)}{\alpha}\int_{\vert x \vert_a}^{\infty}\widetilde{G}(r)(\widetilde{G_1}(r))^{\frac{-(\beta+s)}{\alpha}-1}W_{0}(r)dr\bigg)^\alpha\nonumber
  \\
&   \le G^{\beta}(x)\bigg(\sup_{\vert y \vert_a >\vert x \vert_a} G^{-s}(y) W_{0}^{\alpha}(y)\bigg)\times
\\ & 
\;\; \times \bigg(\widetilde{G_{1}}^{-\beta/\alpha}{(\infty)} + \frac{(\beta+s)}{\alpha}\int_{\vert x \vert_a}^{\infty}\widetilde{G_1}^{-(\beta/\alpha)-1}(r) \frac{d}{dr}\bigg(\int_0^r\widetilde{G}(s)ds\bigg)dr\bigg)^\alpha\nonumber
\\ &
   \le G^{\beta}(x)\sup_{\vert y \vert_a > 0} {A_{4}}(y;\alpha,\beta,s)\bigg(\widetilde{G_{1}}^{-\beta/\alpha}{(\infty)} + \frac{(\beta+s)}{\beta} \bigg(\widetilde{G_1}^{-(\beta/\alpha)}(\vert x \vert_a) -{\widetilde{G_1}}^{-\beta/\alpha}(\infty)\bigg)\bigg)^\alpha\nonumber
\\
   &= \sup_{\vert y \vert_a > 0} {A_{4}}(y;\alpha,\beta,s) \bigg[ \frac{(\beta+s)}{\beta} + \bigg(1-\frac{(\beta+s)}{\beta}\bigg)\bigg(\frac{G(x)}{G(\infty)}\bigg)^{\beta/\alpha}\bigg ]^\alpha\nonumber
 \\ &
   \le \bigg(1+\frac{s}{\beta}\bigg)^\alpha \sup_{\vert y \vert_a>0} A_4(y;\alpha,\beta,s),
   \end{align*}
  where the expressions like $G(\infty)$ make sense since $g\in L^1(\X)$.
   Therefore, we obtain   
   $$
   \mathcal A_1\le (1+s/\beta)^\alpha \mathcal A_4.
   $$
   To prove the opposite inequality, we assume that
  $$
   \sup_ {\vert x \vert_a>0}{ A_1}(x;\alpha, \beta)<\infty.
  $$
   Then we have   
   \begin{align}
    A_{4}(x;\alpha,\beta,s)&= G^{-s}(x)\bigg(\int_{B(a,\vert x \vert_a)} G^{(\beta+s)/\alpha}(y)f(y)dy\bigg)^\alpha\nonumber
   \end{align}
\begin{align}
&=G^{-s}(x)\bigg(\int_0^{\vert x \vert_a}{\widetilde{ G_1}}^{(\beta+s)/\alpha}(r) \frac{d}{dr}\bigg(-\int_r^{\infty}{\widetilde{F}(s)}ds\bigg)dr\bigg)^\alpha\nonumber
\end{align}
\begin{align}
&=G^{-s}(x)\bigg({\widetilde{ G_1}}^{(\beta+s)/\alpha}(r){\widetilde{F_1}}(r) \bigg\vert_{\vert x \vert_a}^0+\frac{\beta+s}{\alpha}\int_0^{\vert x \vert_a}{\widetilde{F_1}}(r){\widetilde{ G_1}}^{(\beta+s)/\alpha-1}(r) \frac{d}{dr}\bigg(\int_0^r{\widetilde{G}(s)}ds\bigg)dr\bigg)^\alpha\nonumber
\end{align}
\begin{align}
\le G^{-s}(x)\bigg (\sup_{0<r<\vert x \vert_a}{\widetilde{ G_1}}^{\beta}(r){\widetilde{F_1}}^\alpha(r)\bigg)\bigg(\frac{\beta+s}{\alpha}\int_0^{\vert x \vert_a}{\widetilde{ G_1}}^{s/\alpha-1}(r)\frac{d}{dr}(\int_0^r{\widetilde{G}(s)}ds)dr\bigg)^\alpha\nonumber
\end{align}
\begin{align}
\le {\left(\frac{\beta+s}{\alpha}\right)}^\alpha \sup _{\vert y \vert_a>0}G^\beta(y) F^\alpha(y) G^{-s}(x)\bigg(\frac{\alpha}{s} G^{s/\alpha}(x)\bigg)^\alpha\nonumber
\end{align}
\begin{align}
&= {\left(\frac{\beta+s}{s}\right)}^\alpha \sup_{\vert x \vert_a>0} A_1(x;\alpha,\beta),\nonumber
\end{align}
where we have used that $f\in L^1(\X)$. Hence we have proved that 
$A_1\approx A_4.$

The proof of $\boxed{A_1\approx A_5}$ follows
 the same lines as that of the case $A_1\approx A_4$ if we interchange the roles of $F$ and $G$.
\end{proof}

\section{Equivalent conditions for the Hardy inequality}
\label{SEC:proofH}

In this section we prove Theorem \ref{THM:Hardy1} and also give some comments concerning Theorem \ref{THM:Hardy2}.
Without loss of generality we can assume that $f\geq 0$.
Then we observe that if in Theorem \ref{THM:equivalence} we take 
$$f(x)=u(x), \; g(x)=v^{1-p'}(x), \;\alpha=\frac{1}{q}, \; \beta=\frac{1}{p'},$$ 
then it follows that we have the equivalence of the quantities 
$$
                          \mathcal D_{1} \approx \mathcal D_{2} \approx \mathcal D_{3} \approx \mathcal D_{4} \approx \mathcal D_{5}.
$$                   
   So, we first assume that any one of these equivalent conditions holds true. In particular,  $\mathcal D_1<\infty$, and using  polar coordinates \eqref{EQ:polar}, we have for every $a>0$ that
 \begin{equation}\label{EQ:D1-cond}
 \bigg\{\int_a^{\infty}\int_{\Sigma_r}
 \lambda(r,\omega)u(r,\omega)d\omega_r dr\bigg\}^\frac{1}{q}\bigg\{\int_0^{a}\int_{\Sigma_r}\lambda(r,\omega)v^{1-p'}(r,\omega)d\omega_r dr\bigg\}^\frac{1}{p'}\le \mathcal D_1.
 \end{equation}
We denote
$$
     h(t):=\bigg( \int_0^t \int_{\Sigma_s} \lambda(s,\sigma) v^{1-p'}(s,\sigma)ds d\sigma_s\bigg)^\frac{1}{pp'},
$$
     \begin{align}
     \widetilde{U}_1(t):= \int_{\Sigma_t}  \lambda(t,\omega) u(t,\omega)d\omega_t\nonumber,
     \end{align}   
     \begin{align}
       F_1(s):= \int_{\Sigma_s} \lambda(s,\sigma) [f(s,\sigma) v^{\frac{1}{p}}(s,\sigma)h(s)]^p d\sigma_s\nonumber,
     \end{align}     
    and                  
     \begin{align}
     H_1(t):=\int_0^t \int_{\Sigma_s} \lambda(s,\sigma)[v^\frac{1}{p}(s,\sigma)h(s)]^{-p'}d\sigma_s ds.\nonumber
     \end{align}
Then using polar coordinates  \eqref{EQ:polar}, H\"older's inequality, and Minkowski's inequality, the left side of \eqref{EQ:Hardy1} can be estimated as 
   \begin{align}\label{EQ:e0}
  & \int_\mathbb X u(x) \bigg(\int_{B(a,\vert x \vert_a)}f(y)dy\bigg)^qdx \nonumber
  \\ &\leq\int_0^{\infty}\int_{\Sigma_r}\lambda(r,\omega)u(r,\omega)
   \bigg( \int_0^r \int_{\Sigma_s}  \lambda(s,\sigma)[f(s, \sigma)v^{\frac{1}{p}}(s,\sigma)h(s)]^pdsd\sigma_s \bigg)^\frac{q}{p}\nonumber
   \\
   &\quad\times \bigg(\int_0^r \int_{\Sigma_s} \lambda(s,\sigma) [v^{\frac{1}{p}}(s,\sigma) h(s)]^{-p'}ds d\sigma_s \bigg)^\frac{q}{p'}d\omega_r dr\nonumber
 \\
     &=\int_0^\infty \widetilde{U}_1(r)\bigg(\int_0^r F_1(s)ds\bigg)^\frac{q}{p}H_1^\frac{q}{p'}(r)dr\nonumber
   \\
  &   \le \bigg(\int_0^\infty F_1(s) \bigg(\int_s^\infty \widetilde{U}_1(r) H_1^\frac{q}{p'}(r)dr\bigg)^\frac{p}{q}ds\bigg)^\frac{q}{p}.
    \end{align}                   
 Denoting 
$$
     V_1(s):= \int_{\Sigma_s}  \lambda(s,\sigma) v^{1-p'}(s,\sigma)d\sigma_s,
$$
     and using inequality \eqref{EQ:D1-cond} we can estimate
     \begin{align}\label{EQ:e1}
    & H_1(t) =\int_0^t \int_{\Sigma_r} \lambda(r,\sigma)[v^\frac{1}{p}(r,\sigma)h(r)]^{-p'}d\sigma_r dr\nonumber
     \\
     &=\int_0^t \int_{\Sigma_r} \lambda(r,\sigma) v^{1-p'}(r,\sigma)\bigg( \int_0^r \int_{\Sigma_{\rho}} \lambda(\rho,\omega) v^{1-p'}(\rho,\omega)d\rho d\omega_{\rho}\bigg)^{-\frac{1}{p}}dr d\sigma_r\nonumber
     \\
     &=\int_0^t V_1(r)\bigg(\int_0^r  V_1(\rho) d\rho \bigg)^{-\frac{1}{p}} dr\nonumber
\\
&=p'\bigg(\int_0^t V_1(r)dr \bigg)^\frac{1}{p'}\nonumber
\\
&=p'\bigg(\int_0^t \int_{\Sigma_r} \lambda(r,\sigma) v^{1-p'}(r,\sigma)dr d\sigma_r \bigg)^\frac{1}{p'}\bigg(\int_t^\infty \int_{\Sigma_r} \lambda(r,\sigma) u(r,\sigma)dr d\sigma_r \bigg)^\frac{1}{q}\nonumber
\\ 
&\quad\times \bigg(\int_t^\infty \int_{\Sigma_r} \lambda(r,\sigma) u(r,\sigma)dr d\sigma_r \bigg)^{-\frac{1}{q}}  \nonumber
\\
& \le p' \mathcal D_1 \bigg(\int_t^\infty\widetilde{U}_1(s)ds \bigg)^{-\frac{1}{q}}.
     \end{align}
     At the same time we can also estimate
     \begin{align}\label{EQ:e2}
     \int_s^\infty \widetilde{U}_1(t) \bigg(\int_t^\infty \widetilde{U}_1(\tau)d\tau\bigg)^{-\frac{1}{p'}} dt
    &=-p \int_s^\infty {\frac{d}{dt}} \bigg(\int_t^\infty \widetilde{U}_1(\tau)d\tau\bigg)^ \frac{1}{p} dt\nonumber 
     \\
     &=p \left(\int_s^\infty \widetilde{U}_1(t)dt \right)^\frac{1}{p}\nonumber
     \\
     &=p \left(\int_s^ \infty \int_{\Sigma_t} \lambda(t,\omega) u(t,\omega)dt d\omega_t\right)^\frac{1}{p}\nonumber
     \\
     &= p \left\{\left(\int_s^ \infty \int_{\Sigma_t}\lambda(t,\omega) u(t,\omega)dtd\omega_t\right)^\frac{1}{q} \right.\nonumber
     \\
     &\quad\times \left.\left(\int_0^s \int_{\Sigma_t} \lambda(t,\omega)v^{1-p'}(t,\omega)dtd\omega_t\right)^\frac{1}{p'} \right\} ^\frac{q}{p}\nonumber
     \\
     &\quad\times \left(\int_0^s\int_{\Sigma_t} \lambda(t,\omega)v^{1-p'}(t,\omega)dtd\omega_t\right)^{-\frac{q}{p'p}}\nonumber
     \\
     &\le p \mathcal D_1 ^\frac{q}{p} h^{-q}(s), 
     \end{align}
     in view of \eqref{EQ:D1-cond}.
     Therefore using \eqref{EQ:e1} and \eqref{EQ:e2} in \eqref{EQ:e0}, we have
     \begin{align}
     \int_\mathbb X u(x) \bigg(\int_{B(a,\vert x \vert_a)} f(y) dy \bigg)^qdx\leq \mathcal D_1^q {p'}^\frac{q}{p'}p \bigg(\int_\mathbb X v(x) {f(x) }^p dx \bigg)^\frac{q}{p}.\nonumber
     \end{align}
     Hence, it follows that \eqref{EQ:Hardy1} holds with
     $C\leq \mathcal D_1(p')^{\frac{1}{p'}} p^\frac{1}{q}$ proving one of the relations in \eqref{EQ:constants}.
  
     Conversely, let us assume that inequality \eqref{EQ:Hardy1} holds, and consider the function
     $$f(x)=v^{1-p'}(x) \chi_{(0,t)} (\vert x \vert_a)$$
     for some $ t>0$, and where $\chi$ is the cut-off function.
     With this function, the right hand side of of \eqref{EQ:Hardy1} takes the form    
        \begin{align}
       &\bigg(\int_{\mathbb X} v(x){ \vert f(x) \vert }^pdx\bigg)^\frac{1}{p}
       =  \bigg(\int_{\vert x \vert_a \le t} v^{1-p'}(x)dx\bigg)^\frac{1}{p}.\nonumber
\end{align}
At the same time, the left hand side of of \eqref{EQ:Hardy1} takes the form
\begin{align}
\bigg(\int_\mathbb X\bigg(\int_{B(a,\vert x \vert_a)}\vert f(y) \vert dy\bigg)^q u(x)dx\bigg)^\frac{1}{q}
&\geq \bigg(\int_{\vert x \vert_a \geq t} \bigg(\int_{B(a,\vert x \vert_a)}\vert f(y) \vert dy\bigg)^q u(x)dx\bigg)^\frac{1}{q}\nonumber
\\
&=\bigg(\int_{\vert x \vert_a \geq t} \bigg(\int_{\vert y \vert_a \leq t} v^{1-p'}(y) dy\bigg)^q u(x)dx\bigg)^\frac{1}{q}\nonumber
\\
&=\bigg(\int_{\vert x \vert_a \geq t}  u(x)dx\bigg)^\frac{1}{q} \bigg(\int_{\vert y \vert_a \leq t} v^{1-p'}(y) dy\bigg).\nonumber
\end{align}
Altogether the inequality  \eqref{EQ:Hardy1} takes the form
\begin{align}
     \bigg(\int_{\vert x \vert_a \ge t} u(x)dx\bigg)^\frac{1}{q}\bigg(\int_{\vert y \vert_a \le t} v^{1-p'}(y)dy \bigg)\le C \bigg(\int_{\vert x \vert_a \le t} v^{1-p'}(x)dx\bigg)^\frac{1}{p},\nonumber
     \end{align}
     which gives 
     $\mathcal D_1\le C$.
     Hence, we have the equivalence and the second relation in  \eqref{EQ:constants}.

\smallskip
  As for Theorem \ref{THM:Hardy2}, 
if we take $f(x)= v^{1-p'}$ , $g(x)=u(x)$, $\alpha= \frac{1}{p'}$  and  $\beta= \frac{1}{q}$ in Theorem \eqref{THM:equivalence}, we find that 
$$
\mathcal D_{1}^{*} \approx \mathcal D_{2}^{*} \approx \mathcal D_{3}^{*} \approx \mathcal D_{4}^{*} \approx \mathcal D_{5}^{*}.$$

Consequently, we can show Theorem \ref{THM:Hardy2} by the argument similar to that in Section \ref{SEC:proofH} where Theorem \ref{THM:Hardy1} was proved. 
We also note that in the case of homogeneous groups, we can actually also derive it from 
Theorem \ref{THM:Hardy1} by the involutive change of variables $x\mapsto x^{-1}.$

\medskip \noindent
{\bf Data accessibility.} No new data was collected or generated during the course of research.

\medskip \noindent
{\bf Competing interests.} We have no competing interests.

\medskip \noindent
{\bf Authors' contributions.} The authors contributed equally to this paper.

\medskip \noindent
{\bf Acknowledgements.}  The authors would like to thank Nurgissa Yessirkegenov for discussions and for checking our calculations. We would also like to thank all 5 referees of this paper for useful comments.

\medskip \noindent
{\bf  Funding statement.}
The first author was supported in parts by the FWO Odysseus Project, EPSRC
grant EP/R003025/1 and by the Leverhulme Grant RPG-2017-151. 

\medskip \noindent
{\bf Ethics statement.} The work did not involve any collection of human data.

%\bibliographystyle{alphaabbr}
%\bibliography{Suragan-BIB-17-11-19-DV}
%\end{document}

   \end{document}